\documentclass[10pt]{article}

\usepackage[a4paper,left=2.2cm,right=2.2cm,top=2.2cm,bottom=2.2cm]{geometry}
\usepackage[T1]{fontenc}
\usepackage[utf8]{inputenc}
\usepackage{amsmath,amssymb,amsthm}
\usepackage{mathtools}
\usepackage{authblk}
\usepackage[hidelinks]{hyperref}
\hypersetup{
  colorlinks=true,
  linkcolor=black,
  urlcolor=black,
  anchorcolor=black,
  citecolor=black
}

\allowdisplaybreaks[4]

\theoremstyle{plain}
\newtheorem{theorem}{Theorem}
\newtheorem{lemma}[theorem]{Lemma}
\newtheorem{proposition}[theorem]{Proposition}

\theoremstyle{definition}
\newtheorem{definition}[theorem]{Definition}
\newtheorem{assumption}[theorem]{Assumption}

\newcommand{\dif}{\mathop{}\!\mathrm{d}}
\newcommand{\cA}{\mathcal{A}}
\newcommand{\mN}{\mathbb{N}}

\title{Moderate deviation principles for stochastic 2D hydrodynamics-type systems with multiplicative jump noise}
\author[1]{Yue Li
}
\author[2]{Shijie Shang}
\author[3]{Jian Wang}
\affil[1]{School of Mathematics and Statistics, Nanjing University of Science and Technology, Nanjing, Jiangsu 210094, China}
\affil[2]{School of Mathematical Sciences, University of Science and Technology of China, Hefei, Anhui 230026, China}
\affil[3]{School of Mathematics, Hangzhou Normal University, Hangzhou, Zhejiang 311121, China}
\date{\today}

\begin{document}
\maketitle

{\renewcommand{\thefootnote}{}
\footnotetext{\textit{Email:} \href{mailto:yueli@njust.edu.cn}{\texttt{yueli@njust.edu.cn}} (Y. Li); \href{mailto:sjshang@ustc.edu.cn}{\texttt{sjshang@ustc.edu.cn}} (S. Shang); \href{mailto:20230078@hznu.edu.cn}{\texttt{20230078@hznu.edu.cn}} (J. Wang).}
}
\vspace{-2em}
\begin{abstract}
This paper establishes a moderate deviation principle for a class of stochastic 2D hydrodynamical-type systems driven by multiplicative jump noise. The proof does not require compactness of the embedding in the associated Gelfand triple, so the result applies to both bounded and unbounded domains. The combination of finite-dimensional projections and integration by parts is used to prove the strong continuity of the skeleton solution map with respect to weakly convergent controls. This approach avoids the time discretization and intricate jump estimates used in earlier treatments of noncompact settings.

\end{abstract}

\section{Introduction}
This paper considers a class of stochastic hydrodynamical-type systems driven by multiplicative jump noise, including stochastic 2D Navier--Stokes equations, 2D magneto-hydrodynamic (MHD) equations, 3D Leray $\alpha$-models and shell models of turbulence, via the following abstract evolution equation:
\begin{equation}\label{Eq}
	\dif u^\varepsilon(t)+\cA u^\varepsilon(t)\dif t+B(u^\varepsilon(t),u^\varepsilon(t))\dif t= \varepsilon\int_{Z}G(t,u^\varepsilon(t-),z)\widetilde{N}^{\varepsilon^{-1}}(\dif z,\dif t),\quad u^\varepsilon(0)=u_0,
\end{equation}
where $\widetilde{N}^{\varepsilon^{-1}}$ is a compensated Poisson random measure with intensity measure $\varepsilon^{-1}\nu_T$. Here, $\nu_T:=\mathrm{Leb}_T\otimes\nu$, $\nu$ is a $\sigma$-finite measure on the locally compact Polish space $Z$, and $\mathrm{Leb}_T$ denotes Lebesgue measure on $[0,T]$. Equation \eqref{Eq} is formulated on a Gelfand triple $V\subseteq H\subseteq V'$, and no compactness of the embedding $V\subseteq H$ is assumed. The precise definitions and assumptions are given in Section 2.  

\vspace{1em}
As $\varepsilon\to 0$, the solution $u^\varepsilon$ of equation \eqref{Eq} is expected to converge to the solution of the deterministic limiting equation
\begin{eqnarray}\label{eq X0}
	\frac{\dif u^0(t)}{\dif t}+\cA u^0(t)+B(u^0(t),u^0(t))=0,\quad u^0(0)=u_0.
\end{eqnarray}
The present paper is concerned with the moderate deviation principles (MDPs) of $u^\varepsilon$ from $u^0$, namely, the asymptotic behavior of the normalized deviation
\[
M^\varepsilon := \frac{u^\varepsilon - u^0}{a(\varepsilon)},
\]
where the deviation scale $a(\varepsilon)$ satisfies
\begin{equation}\label{scale}
	a(\varepsilon)\to 0,\quad \varepsilon/a^2(\varepsilon)\to 0.
\end{equation} 
This scaling regime lies between the central-limit scale $a(\varepsilon)=\sqrt{\varepsilon}$ and the large-deviation scale $a(\varepsilon)=1$. MDPs characterize the exponential decay rate of deviations at this scale, and are of broad interest in areas such as statistical inference.

\vspace{1em}
The compactness of the embedding $V\subseteq H$ is frequently assumed in the variational framework of SPDEs and also plays an essential role in the weak-convergence approach to large and moderate deviations. Existing LDP and MDP results for Lévy-driven hydrodynamical systems \cite{ZZ15,DXZZ17,BPZ23,DS23} rely essentially on this compactness, both to establish tightness of the controlled equations and to prove the strong continuity of the skeleton solution map. However, such compactness generally fails for hydrodynamical equations on unbounded spatial domains, or in more general abstract settings. We establish the MDP for \eqref{Eq} without assuming that $V\subseteq H$ is compact. The result thus covers both compact and noncompact settings and, in particular, includes the bounded-domain result of \cite{DXZZ17} as a special case.

\vspace{1em}
To handle the lack of compactness, we develop a direct argument. Existing methods for noncompact LDPs use time discretization at the cost of either additional assumptions or intricate jump estimates. In the Gaussian case, \cite{CM10} imposes an additional Hölder continuity assumption in time on the diffusion coefficient. For Lévy noise, \cite{WZ24,WZZ25,ZLZ26} require separating small and large jumps and estimating numerous terms; see, for example, (5.8)--(5.10) in \cite{WZZ25}. Our approach instead combines finite-dimensional projections with integration by parts. 

\vspace{1em}
More precisely, to prove the strong continuity of the skeleton solution map, we project a suitably constructed integral term onto a finite-dimensional subspace $H_k\subseteq H$ and use the Arzelà--Ascoli theorem to obtain its convergence to zero for each fixed $k$. Integration by parts then allows us to estimate the relevant terms through the projected integral, while the projection error is controlled by the $L^2$-integrability of $G$; see Proposition~\ref{MDP1}. For the stochastic controlled equations, we avoid tightness arguments and prove directly that the difference between the controlled and skeleton solutions converges to zero in mean square; see Proposition~\ref{MDP2}. 
Moreover, we do not require fourth-moment bounds \cite{BHZ13}, small-jump restrictions \cite{DX09}, or exponential integrability \cite{DS23} for the jump coefficient. Finally, although developed here for \eqref{Eq}, its main idea may also be adapted to other SPDEs without compact embeddings, with jump or Gaussian noise.

\vspace{1em}
For $\varphi\in L^2(\nu_T)$ (the space of square-integrable functions on $[0,T]\times Z$ with respect to $\nu_T$, with norm denoted by $\|\cdot \|_2$), the skeleton equation is given by
\begin{equation}\label{eq sk}
	\frac{\dif Y^\varphi(t)}{\dif t}+\mathcal{A}Y^\varphi(t)+B(Y^\varphi(t),u^0(t))+B(u^0(t),Y^\varphi(t))=\int_Z G(t,u^0(t),z)\varphi(t,z)\nu(\dif z),
\end{equation}
with initial value $Y^\varphi(0)=0$. 

\vspace{1em}
Our main result is the following MDP for the solution $u^\varepsilon$ of equation \eqref{Eq}, equivalently, an LDP for $M^\varepsilon$ with $a(\varepsilon)$ as in \eqref{scale}; see, for example, \cite{DZ98}. The definition of an LDP is recalled in Section 2.
\begin{theorem}\label{Thm MDP}
Suppose that \textbf{Assumptions \ref{p-2} and \ref{con G}} below hold and that $u_0\in H$. Then the family $\left\{M^\varepsilon=\frac{u^\varepsilon-u^0}{a(\varepsilon)}\right\}_{\varepsilon>0}$ satisfies an LDP in $D([0,T];H)\cap L^2([0,T];V)$ as $\varepsilon\to 0$, with speed $\varepsilon/a^2(\varepsilon)$ and rate function
	$$
	I(\eta)=\inf_{\varphi\in L^2(\nu_T):\;\eta=Y^\varphi}\Big\{\frac{1}{2}\|\varphi\|^2_2\Big\},  \ \ \ \ \ \forall\eta\in D([0,T];H)\cap L^2([0,T];V),
	$$
	with the convention $\inf\emptyset=\infty$. Here $D([0,T];H)$ denotes the space of $H$-valued càdlàg functions endowed with the Skorokhod topology.
\end{theorem}

The remainder of the paper is organized as follows. Section 2 presents the preliminaries, assumptions and some well-posedness results. Section 3 is devoted to the proof of Theorem \ref{Thm MDP}, with the finite-dimensional projection and integration-by-parts arguments presented in Subsection 3.1.

Throughout the paper, $c$ and $C$ denote generic positive constants whose values may change from line to line.

\section{Preliminaries and hypotheses}\label{S:2}
This section collects the preliminaries, the assumptions on the coefficients $B$ and $G$, and the well-posedness results needed in the proof.

\subsection{Preliminaries}
\noindent\textbf{Functional setting.} Let $H$ be a separable Hilbert space with inner product $\langle\cdot,\cdot\rangle$ and norm $|\cdot|$. The operator$\cA$ is a (possibly unbounded) self-adjoint positive definite linear operator on $H$. Set $V:=\rm{dom} (\cA^{1/2})$  and equip it with the norm $\|x\|:=|\cA^{\frac{1}{2}}x|,\ x\in V.$
Denote by $V'$ the dual of $V$. Identify $H$ with $H'$, we obtain the Gelfand triple $$V\subseteq H\subseteq V'.$$
The duality between $V'$ and $V$ is also denoted by $\langle \cdot,\cdot\rangle$.

\vspace{1em}
Recall from Section 1 that $Z$ is a locally compact Polish space equipped with a $\sigma$-finite measure $\nu$ and that $\nu_T=\text{Leb}_T\otimes \nu$. We denote by $L^2(\nu_T)$ the associated Hilbert space of square-integrable, real-valued functions on $[0,T]\times Z$, with norm $\|h\|_2$ defined by 
\[
\|h\|_2^2 := \int_0^T \int_Z |h(t,z)|^2 \, \nu(\mathrm{d}z) \mathrm{d}t < \infty, 
\]
The closed ball in $L^2(\nu_T)$ of radius $r > 0$ is denoted by $B_2(r)$, equipped with the weak topology (in which it is compact).

\vspace{1em}
\noindent\textbf{Controlled Poisson random measures.} Recall that $\widetilde{N}^{\varepsilon^{-1}}$ is the compensated Poisson random measure on $[0,T]\times Z$ with intensity measure $\varepsilon^{-1} \nu_T$, as introduced in Section 1. 
To show the MDP, specially condition (C2) in Section 3, we replace the fixed intensity $\varepsilon^{-1}$ by a predictable intensity $\varphi$. This leads to the controlled Poisson random measures $N^\varphi$ and their compensated measures $\widetilde{N}^\varphi$. 
We recall their standard construction on an extended space; see, for example, \cite{BDG16}.

Let $$\mathbb{M} := M_{FC}([0,T]\times Z \times [0,\infty))$$ be the space of Borel measures on $[0,T]\times Z \times [0,\infty)$ that are finite on compact sets, endowed with the usual vague topology, which makes $\mathbb{M}$ a Polish space. 
Denote by $\mathbb{P}$ the unique probability measure on $(\mathbb{M}, \mathcal{B}(\mathbb{M}))$ under which the canonical mapping $$\bar{N}:\mathbb{M}\rightarrow \mathbb{M},\qquad \bar{N}(\omega):=\omega,$$ is a Poisson random measure with intensity measure $\text{Leb}_T\otimes\nu\otimes\text{Leb}_\infty $, where $\text{Leb}_\infty $ denotes the Lebesgue measure on $[0,\infty)$. The corresponding compensated Poisson random measure is denoted by $\widetilde{\bar{N}}$. 
Let $\mathbb{F}:=\{\mathcal{F}_t\}_{t\in[0,T]}$ be the augmentation of the filtration generated by $\bar{N}$ under $\mathbb{P}$. 
This gives the filtered probability space $(\mathbb{M},\mathcal{B}(\mathbb{M}),\mathbb{F},\mathbb{P})$.

Denote by $\mathcal{P}$ the predictable $\sigma$-field on $[0,T]\times\mathbb{M}$ with respect to the filtration $\mathbb{F}$. A mapping on $[0,T]\times Z\times\mathbb{M}$ is called predictable if it is measurable with respect to $\mathcal{B}(Z)\otimes\mathcal{P}$. 
For a predictable mapping $\varphi:[0,T]\times Z\times\mathbb{M}\to [0,\infty)$, define the random measure $N^\varphi$ on $[0,T]\times Z$ by
\begin{equation}\label{Jump-representation}
	N^\varphi((0,t]\times A):=\int_{(0,t]\times A\times[0,\infty)}
	\mathbf{1}_{[0,\varphi(s,z)]}(r)\,\bar{N}(\mathrm{d}s,\mathrm{d}z,\mathrm{d}r),
	\qquad 0\le t\le T,\; A\in\mathcal{B}(Z).
\end{equation}
The compensator of $N^\varphi$ is $\varphi(s,z)\nu(\mathrm{d}z)\mathrm{d}s$, so $\varphi$ controls its intensity. 
Analogously, $\widetilde{N}^\varphi$ is defined by replacing $\bar{N}$ with $\widetilde{\bar{N}}$ in \eqref{Jump-representation}. When $\varphi(s,z,\omega)\equiv\varepsilon^{-1}$, we write $N^\varphi=N^{\varepsilon^{-1}}$ and $\widetilde{N}^\varphi=\widetilde{N}^{\varepsilon^{-1}}$.

Under $\mathbb{P}$, $N^{\varepsilon^{-1}}$ is a Poisson random measure on $[0,T]\times Z$ with intensity measure $\varepsilon^{-1}\nu_T$, and $\widetilde{N}^{\varepsilon^{-1}}$ is its compensated measure.  Thus the compensated Poisson random measure in equation \eqref{Eq} is realized on the filtered probability space constructed above.

\vspace{1em}
\noindent\textbf{Large deviation principle.} For the readers' convenience, we recall the definition of an LDP for a family of random elements $\left\{X^{\varepsilon}\right\}_{\varepsilon>0}$ taking values in  a Polish space $\mathcal{E}$.

\begin{definition}[Large deviation principle] Let $I: \mathcal{E} \rightarrow[0, \infty]$ be a good rate function on $\mathcal{E}$, that is, for each $M<\infty$, the level set $\{x \in \mathcal{E}: I(x) \leq M\}$ is compact in $\mathcal{E}$. 
Given positive numbers $\{\hbar(\varepsilon)\}_{\varepsilon>0}$ such that $\hbar(\varepsilon)\to0$,
a family $\left\{X^{\varepsilon}\right\}_{\varepsilon>0}$ of $\mathcal{E}$-valued random elements is said to satisfy an LDP on $\mathcal{E}$ with speed $\hbar(\varepsilon)$ and rate function $I$ if the following two bounds hold.
\begin{itemize}
  \item[(a)]
 (Upper bound) For each closed subset $F$ of $\mathcal{E},$
$$
\limsup_{\varepsilon \rightarrow 0} \hbar(\varepsilon) \log {\mathbb{P}}\left(X^{\varepsilon} \in F\right) \leq - \inf_{x \in F} I(x).
$$
 \item[(b)] (Lower bound) For each open subset $O$ of $\mathcal{E}$,
$$
\liminf_{\varepsilon \rightarrow 0} \hbar(\varepsilon) \log {\mathbb{P}}\left(X^{\varepsilon} \in O\right) \geq - \inf_{x \in O} I(x).
$$
\end{itemize}
\end{definition}

\subsection{Hypotheses}
Our assumptions on the coefficient $B$ are as follows. 
\begin{assumption}\label{p-2}
Assume that the mapping $B:V\times V\rightarrow V'$ is a continuous bilinear mapping satisfying the following conditions:
\begin{itemize}
\item[(B1)](skew-symmetry)
\begin{eqnarray}\label{p-3}
  \langle B(u_1,u_2),u_3\rangle=-\langle B(u_1,u_3),u_2\rangle, \text{ for all } u_1,u_2,u_3\in V,
\end{eqnarray}
which implies that \begin{eqnarray}\label{eq PB}
	\langle B(u_1,u_2),u_2\rangle=0, \text{ for all } u_1,u_2\in V.
\end{eqnarray}

\item[(B2)] There exist a reflexive, separable Banach space $(Q, |\cdot|_Q)$ and a constant $a_0>0$ such that
    \begin{eqnarray}
      && V\subseteq Q \subseteq H, \nonumber\\
      \label{p-4} && |v|_Q^2\leq a_0|v| \|v\|, \text{  for all } v\in V.
    \end{eqnarray}
\item[(B3)] There exists a constant $C>0$ such that
\begin{eqnarray}
\label{p-5}
  |\langle B(u,v), w\rangle|\leq C|u|_Q \|v\| |w|_Q,
  \text{ for all } u,v,w\in V.
\end{eqnarray}
\end{itemize}
\end{assumption}

\noindent\textbf{Remark.}
The hypotheses imposed on $\cA$ and $B$ are satisfied by a class of 2D hydrodynamic-type models, including the 2D Navier--Stokes and 2D magneto-hydrodynamic (MHD) equations. They also cover some regularized higher-dimensional models, such as the 3D Leray $\alpha$-model, and certain shell models of turbulence. We refer to \cite[Section 2]{CM10} for the corresponding verifications.

\vspace{1em}
To formulate the assumptions on $G$, we introduce the auxiliary function class $\mathcal{H}^\rho$. For $\rho>0$, define 
\begin{align*}
\mathcal{H}^\rho :=
\Big\{
                 h:[0,T]\times Z\rightarrow\mathbb{R}_+ \;\big|\;\int_0^T\int_Z e^{\rho h(t,z)}\mathbf{1}_{U}(t,z)\nu(dz)dt<\infty \\
                 \text{ for any }
                 U\in \mathcal{B}([0,T])\otimes\mathcal{B}(Z) \text{ with } \nu_T(U) <\infty 
                 \Big\} .
\end{align*}
\begin{assumption}\label{con G}
Let $G:[0,T]\times H\times Z\rightarrow H$ be a measurable mapping.  There exist a constant $\rho>0$ and functions $L_i\in \mathcal{H}^\rho\cap L^2(\nu_T),\ i=1,2,3$, such that, for all $t\in[0,T],\ v, v_1, v_2\in H$, and $\nu$-a.e. $z\in Z$,
\begin{itemize}
\item[(H1)](Lipschitz) $$
             |G(t,v_1,z)-G(t,v_2,z)|\leq L_1(t,z)|v_1-v_2|;
            $$
\item[(H2)](Growth) $$
             |G(t,v,z)|\leq L_2(t,z)+L_3(t,z)|v|.
            $$
\end{itemize}
\end{assumption}

\subsection{Well-posedness results}

We next collect the well-posedness results for the stochastic equation \eqref{Eq}, the limiting deterministic equation \eqref{eq X0}, and the skeleton equation \eqref{eq sk}.
We first define a solution of equation \eqref{Eq}.
%

\begin{definition}
An $H$-valued c\`adl\`ag $\mathbb{F}$-adapted process $u^\varepsilon=(u^\varepsilon(t))_{t\in[0,T]}$ is called a solution of equation \ \eqref{Eq} if the following conditions hold:
\begin{enumerate}
    \item[(S1)] $u^\varepsilon\in D([0,T];H)\cap L^2([0,T];V)$, $\mathbb{P}$-a.s.;
    \item[(S2)] For every $t\in[0,T]$, the following equality holds in $V'$, $\mathbb{P}$-a.s.:
    \begin{equation*}
        u^\varepsilon(t)=u_0-\int_0^t\mathcal{A}u^\varepsilon(s)\,\mathrm{d}s
        -\int_0^t B(u^\varepsilon(s),u^\varepsilon(s))\,\mathrm{d}s
        +\varepsilon\int_0^t\int_Z G(s,u^\varepsilon(s-),z)\,\widetilde{N}^{\varepsilon^{-1}}(\mathrm{d}z,\mathrm{d}s).
    \end{equation*}
\end{enumerate}
\end{definition}

\begin{proposition}\phantomsection\label{thm solution}
\begin{enumerate}
    \item[(i)] 
    Suppose that Assumptions \ref{p-2} and \ref{con G} hold and that $u_0\in H$. Then there exists $\varepsilon_0\in(0,1)$ such that for every $\varepsilon\in(0,\varepsilon_0]$, equation \eqref{Eq} has a unique solution $u^\varepsilon=(u^\varepsilon(t))_{t\in[0,T]}$. Moreover,
    \begin{equation*}
        \sup_{t\in[0,T]}\mathbb{E}\left[|u^\varepsilon(t)|^2\right]
        +\mathbb{E}\left[\int_0^T \|u^\varepsilon(t)\|^2\,\mathrm{d}t\right]<\infty.
    \end{equation*}

    \item[(ii)] 
    Suppose that Assumption \ref{p-2} holds and that $u_0\in H$. Then equation \eqref{eq X0} has a unique solution $u^0\in C([0,T];H)\cap L^2([0,T];V)$. Moreover,
    \begin{equation}\label{esti X0}
        \sup_{t\in[0,T]}|u^0(t)|^2+2\int_0^T \|u^0(t)\|^2\,\mathrm{d}t\le |u_0|^2.
    \end{equation}

    \item[(iii)] 
    Suppose that Assumptions \ref{p-2} and \ref{con G} hold and that $u_0\in H$. Then for every $\varphi\in L^2(\nu_T)$, equation \eqref{eq sk} has a unique solution $Y^\varphi\in C([0,T];H)\cap L^2([0,T];V)$. Moreover,
    \begin{equation*}
        \sup_{t\in[0,T]}|Y^\varphi(t)|^2+\int_0^T \|Y^\varphi(t)\|^2\,\mathrm{d}t
        \le c(\|\varphi\|_2),
    \end{equation*}
    where $c(\|\varphi\|_2)$ is a constant depending on $\|\varphi\|_2$.
\end{enumerate}
\end{proposition}

\begin{proof}
Assertions (i) and (ii) follow respectively from \cite[Theorem 2.4]{PYZ22} and 
\cite[Chapter III, Theorems 3.1 and 3.2, Remark 3.2]{Te79}. We prove (iii) briefly below.
By (H2), 
$$
\int_0^T\int_Z|G(t,u^0(t),z)|^2\nu(\dif z)\dif t \leq C,
$$
so the Cauchy--Schwarz inequality gives $G(\cdot,u^0(\cdot),\cdot)\varphi(\cdot,\cdot)\in L^1([0,T];H)$:
\begin{equation*}
    \int_0^t \int_Z |G(t,u^0(t),z)\varphi(t,z)|\nu(\dif z)\dif t\leq \left(\int_0^T\int_Z|G(t,u^0(t),z)|^2\nu(\dif z)\dif t\right)^{\frac{1}{2}}\|\varphi\|_2<\infty.
\end{equation*}
By \cite[Chapter III, Remark 3.1]{Te79}, the skeleton equation \eqref{eq sk} admits a unique
solution $Y^\varphi\in C([0,T];H)\cap L^2([0,T];V)$.
The bound follows from the standard energy estimate: 
testing the equation against $Y^{\varphi}$ and 
using (B2)--(B3) and Gronwall's inequality yield the desired estimate.
\end{proof}

\section{Proof of Theorem \ref{Thm MDP}}\label{S:3}
In this section, we prove Theorem \ref{Thm MDP} by the weak-convergence method. The key step is to
establish the strong continuity of the skeleton solution map with respect to weakly convergent controls without assuming that the embedding $V\subseteq H$ is compact. This is achieved
in Subsection 3.1 by using an approach based on finite-dimensional projections and integration by parts. 
In Subsection 3.2, we prove directly that the difference between the stochastic controlled equation and the associated skeleton equation converges to $0$.

We begin by recalling the weak-convergence framework and introduce the spaces of intensity controls used below.

For a measurable function $g:[0,T]\times Z\to[0,\infty)$, define
\[
Q(g):=\int_{[0,T]\times Z}\ell(g(s,z))\,\nu(\mathrm{d}z)\mathrm{d}s,
\qquad \ell(x):=x\log x-x+1,\ \ell(0):=1.
\]
For $m>0$ and $\varepsilon>0$, define
\begin{align*}
	S^m_{+,\varepsilon} &:=\left\{g:[0,T]\times Z\rightarrow[0,\infty)\;|\;\ Q(g)\leq ma^2(\varepsilon)\right\},\\
	S^m_\varepsilon &:=\left\{\varphi:[0,T]\times Z\rightarrow \mathbb{R}\;|\;\varphi=(g-1)/a(\varepsilon),\ g\in S^m_{+,\varepsilon}\right\}.
\end{align*}
Predictability is understood with respect to $\mathcal{B}(Z)\otimes\mathcal{P}$ (see Section 2).
\begin{align*}
\mathcal{S}^m_{+,\varepsilon}
&:= \left\{g:[0,T]\times Z\times\mathbb{M}\to[0,\infty) \text{ is predictable} \;|\; g(\cdot,\cdot,\omega)\in S^m_{+,\varepsilon}, \text{ for } \mathbb{P}\text{-a.e. }\omega \right\},\\
\mathcal{S}^m_\varepsilon
&:= \left\{\varphi:[0,T]\times Z\times\mathbb{M}\to\mathbb{R}\text{ is predictable} \;|\;  \varphi(\cdot,\cdot,\omega)\in S^m_\varepsilon,\text{ for } \mathbb{P}\text{-a.e. }\omega \right\}.
\end{align*}

Suppose $g\in S^m_{+,\varepsilon}$. By Lemma 3.2 in \cite{BDG16}, there exists a constant $\kappa_2(1)>0$ (independent of $\varepsilon$) such that
$\varphi\mathbf{1}_{\{|\varphi|\leq1/a(\varepsilon)\}}\in B_2(\sqrt{m\kappa_2(1)})$, where $\varphi=(g-1)/a(\varepsilon)$.

Let $M_{FC}([0,T]\times Z)$ denote the space of Borel measures on $[0,T]\times Z$ that are finite on compact sets; recall from Section 2.1 that $N^{\varepsilon^{-1}}$ takes values in this space.

\vspace{1em}
We use the following sufficient condition established in \cite{LSZZ23} for the MDP of $X^\varepsilon$. It is modification of the classical Budhiraja--Dupuis--Ganguly criterion in \cite{BDG16} and is convenient for our setting. 
\begin{lemma}[{\cite[Theorem 3.6]{LSZZ23}}]\label{lem}
Let $\mathcal{E}:=D([0,T];H)\cap L^2([0,T];V)$. Suppose there exist measurable mappings
\[
\Gamma^0: L^2(\nu_T)\to \mathcal{E}, \qquad 
\Gamma^\varepsilon: M_{FC}([0,T]\times Z)\to \mathcal{E},
\]
such that
\[
X^\varepsilon := \Gamma^\varepsilon(\varepsilon N^{\varepsilon^{-1}})
\]
for $\varepsilon>0$, where $X^\varepsilon$ is the family of random elements for which the MDP is to be established.

Assume that the following two conditions hold:

\noindent\textbf{(C1)} For any $m>0$, if $\{\varphi_n\}\subseteq B_2(m)$ and $\varphi_n\to\varphi$ weakly in $L^2(\nu_T)$ as $n\to\infty$, then
\[
\Gamma^0(\varphi_n)\rightarrow \Gamma^0(\varphi)
\quad\text{in }\ \mathcal{E}.
\]

\noindent\textbf{(C2)} For any $m>0$, let $\{\psi_\varepsilon\}_{\varepsilon>0}$ satisfy $\psi_\varepsilon\in\mathcal{S}_{+,\varepsilon}^m$ for every $\varepsilon>0$, and set
\[
\varphi_\varepsilon := \frac{\psi_\varepsilon-1}{a(\varepsilon)},\qquad
\varphi_\varepsilon \mathbf{1}_{\{|\varphi_\varepsilon|\le \beta/a(\varepsilon)\}}\in B_2(\sqrt{m\kappa_2(1)})\quad\text{for some }\beta\in(0,1]. 
\] 
Then
\[
\Gamma^\varepsilon(\varepsilon N^{\varepsilon^{-1}\psi_\varepsilon})
-
\Gamma^0\!\left(\varphi_\varepsilon \mathbf{1}_{\{|\varphi_\varepsilon|\le \beta/a(\varepsilon)\}}\right)
\rightarrow 0\quad \text{ in }\ \mathcal{E}\text{ in probability, as }\varepsilon\to0.
\]
Then the family $\{X^\varepsilon\}$ satisfies a LDP in $\mathcal{E}$ with speed $\varepsilon/a^2(\varepsilon)$ and rate function
\[
I(\eta):=\inf_{\varphi\in L^2(\nu_T):\, \eta=\Gamma^0(\varphi)}
\left\{\frac{1}{2}\|\varphi\|_2^2\right\},
\]
with the convention $\inf\emptyset=\infty$.
\end{lemma}

\vspace{1em}
In our setting, the abstract mappings in Lemma \ref{lem} are realized as follows.

First, Proposition \ref{thm solution}(iii) gives the well-posedness of the skeleton equation \eqref{eq sk}, which allows us to define
\[
\Gamma^0:L^2(\nu_T)\to C([0,T];H)\cap L^2([0,T];V),\qquad
\Gamma^0(\varphi):=Y^\varphi,
\]
where $Y^\varphi$ is the unique solution of \eqref{eq sk}.

Second, the deviation process $M^\varepsilon:=(u^\varepsilon-u^0)/a(\varepsilon)$ satisfies
\begin{equation*}
\begin{aligned}
\mathrm{d}M^\varepsilon(t)
&+\mathcal{A}M^\varepsilon(t)\,\mathrm{d}t
+\frac{1}{a(\varepsilon)}
\Big(B(a(\varepsilon)M^\varepsilon(t)+u^0(t),a(\varepsilon)M^\varepsilon(t)+u^0(t))-B(u^0(t),u^0(t))\Big)\mathrm{d}t\\
&=\frac{\varepsilon}{a(\varepsilon)}\int_Z G(t,a(\varepsilon)M^\varepsilon(t-)+u^0(t-),z)\,
\widetilde{N}^{\varepsilon^{-1}}(\mathrm{d}z,\mathrm{d}t),\quad M^\varepsilon(0)=0.
\end{aligned}
\end{equation*}
By Proposition \ref{thm solution} (i)--(ii) and Yamada-Watanabe Theorem (\cite[Theorem 8]{Z14}), there exists a measurable mapping $\Gamma^\varepsilon:M_{FC}([0,T]\times Z)\to D([0,T];H)\cap L^2([0,T];V)$ such that
\[
M^\varepsilon=\Gamma^\varepsilon(\varepsilon N^{\varepsilon^{-1}}).
\]

For $\psi_\varepsilon\in\mathcal{S}^m_{+,\varepsilon}$, define the controlled process
\[
M^{\psi_\varepsilon}:=\Gamma^\varepsilon(\varepsilon N^{\varepsilon^{-1}\psi_\varepsilon}).
\]
By Girsanov's theorem (see, e.g., \cite[Theorem 6.1]{BPZ23}), $M^{\psi_\varepsilon}$ satisfies the controlled SPDE
\begin{equation}\label{eq MDP 1-second}
\begin{aligned}
\mathrm{d}M^{\psi_\varepsilon}(t)
&+\mathcal{A}M^{\psi_\varepsilon}(t)\,\mathrm{d}t
= -\frac{1}{a(\varepsilon)}
\Big(B(a(\varepsilon)M^{\psi_\varepsilon}(t)+u^0(t),a(\varepsilon)M^{\psi_\varepsilon}(t)+u^0(t))-B(u^0(t),u^0(t))\Big)\mathrm{d}t\\
&+\frac{\varepsilon}{a(\varepsilon)}\int_Z G(t,a(\varepsilon)M^{\psi_\varepsilon}(t-)+u^0(t-),z)\,\widetilde{N}^{\varepsilon^{-1}\psi_\varepsilon}(\mathrm{d}z,\mathrm{d}t)\\
&+\frac{1}{a(\varepsilon)}\int_Z G(t,a(\varepsilon)M^{\psi_\varepsilon}(t)+u^0(t),z)
\Big(\psi_\varepsilon(t,z)-1\Big)\nu(\mathrm{d}z)\,\mathrm{d}t,\quad M^{\psi_\varepsilon}(0)=0.
\end{aligned}
\end{equation}
Here, the last term arises from the change of compensator from $\varepsilon^{-1}\nu_T$ to $\varepsilon^{-1}\psi_\varepsilon\nu_T$.

\vspace{1em}
Therefore, by Lemma \ref{lem}, Theorem \ref{Thm MDP} follows once we verify:
\begin{itemize}
    \item (C1): the strong continuity of the skeleton solution map with respect to weakly convergent controls; see Section 3.1.
    \item (C2): the convergence of the stochastic controlled equation to the associated skeleton equation; see Section 3.2.
\end{itemize}

\subsection{Verification of (C1)}\label{S:4}
In this subsection, we prove the strong continuity of the skeleton solution map with respect to weakly convergent controls. To avoid a compactness assumption on the Gelfand triple, we use an approach based on finite-dimensional projections and integration by parts.
\begin{proposition}\label{MDP1}
	For any $m>0$, let $\varphi_n, \varphi\in B_2(m)$ with $\varphi_n\rightarrow\varphi$ weakly in $L^2(\nu_T)$ as $n\rightarrow\infty$. Then
	$$
	\lim_{n\rightarrow\infty}Y^{\varphi_n}=Y^{\varphi}\quad \text{ in } D([0,T];H)\cap L^2([0,T];V).
	$$
	Here $Y^{\varphi}$ is the unique solution of equation \eqref{eq sk}, and $Y^{\varphi_n}$ is the unique solution of equation \eqref{eq sk} with $\varphi$ replaced by $\varphi_n$.
\end{proposition}
\begin{proof}
Proposition \ref{thm solution}(iii) gives the following uniform bound on $Y^{\varphi_n}$ and $Y^\varphi$ in $C([0,T];H)\cap L^2([0,T];V)$:
\begin{equation}\label{esti sk}
	\sup_{n\in \mN}\left(\sup_{t\in [0,T]}|Y^{\varphi_n}(t)|^2+\int_{0}^{T}\Vert Y^{\varphi_n}(t)\Vert^2\dif t \right)
    +\sup_{t\in [0,T]}|Y^\varphi(t)|^2+\int_{0}^{T}\|Y^\varphi(t)\|^2\dif t\leq  c_1, \\
\end{equation}
where the constant $c_1$ depends on $m$. 

Subtracting the two skeleton equations, testing the difference by $Y^{\varphi_n}-Y^\varphi$, and using \eqref{eq PB}, we obtain, for $t\in[0,T]$,
\begin{eqnarray}\label{MDP1-0}
	&&\frac{1}{2}|Y^{\varphi_n}(t)-Y^\varphi(t)|^2+\int_{0}^{t}\|Y^{\varphi_n}(s)-Y^\varphi(s)\|^2\dif s \nonumber\\
	&= &-\int_{0}^{t}\langle B(Y^{\varphi_n}(s)-Y^\varphi(s), u^0(s)), Y^{\varphi_n}(s)-Y^\varphi(s)\rangle\dif s \nonumber\\
	&&+\int_{0}^{t}\int_{Z}\langle G(s,u^0(s),z)(\varphi_n(s,z)-\varphi(s,z)),Y^{\varphi_n}(s)-Y^\varphi(s)\rangle \nu (\dif z)\dif s.
\end{eqnarray}
The first term on the RHS of \eqref{MDP1-0} is bounded by 
\begin{equation}\label{MDP1-1}
	\frac{1}{2}\int_{0}^{t}\|Y^{\varphi_n}(s)-Y^\varphi(s)\|^2\dif s+C_1\int_{0}^{t}\|u^0(s)\|^2|Y^{\varphi_n}(s)-Y^\varphi(s)|^2\dif s,
\end{equation}
using \eqref{p-5}, \eqref{p-4} and Young's inequality.

We use the combination of finite-dimensional projections and integration by parts to treat the second term. 
Since $V$ is dense in $H$, choose an orthonormal basis $\{e_i\}_{i\in\mathbb{N}}$ of $H$ such that $e_i\in V$ for every $i$. Let $H_k:=\operatorname{span}\{e_1,\ldots,e_k\}$, and let $P_k:H\to H_k$ be the orthogonal projection. For fixed $k\in\mN$, define
%
\begin{equation*}
	\gamma_n^k(t):=\int_{0}^{t}\int_{Z} P_k G(s,u^0(s),z)(\varphi_n(s,z)-\varphi(s,z))\nu(\dif z)\dif s.
\end{equation*}
%
We next show that $\gamma_n^k\rightarrow 0$ in $C([0,T];H)\cap L^2([0,T];V)$ as $n\to\infty$. 
For $0\le t\le t'\le T$, (H2) gives
\begin{eqnarray*}
	|\gamma_n^k(t')-\gamma_n^k(t)|^2_{H_k}&\leq & \int_{t}^{t'}\int_{Z}2(L_2^2(s,z)+L_3^2(s,z)|u^0(s)|^2)\nu(\dif z)\dif s \nonumber\\ &&\times\int_{t}^{t'}\int_{Z}|\varphi_n(s,z)-\varphi(s,z)|^2\nu(\dif z)\dif s,
\end{eqnarray*}
%
%
so $\{\gamma_n^k\}_{n\in\mathbb{N}}$ is equicontinuous on $[0,T]$, because $\{\varphi_n\}$ and $\varphi$ are bounded in $L^2(\nu_T)$.
For every fixed $t\in[0,T]$ and $i\leq k$, the function $(s,z)\mapsto \mathbf{1}_{[0,t]}(s)\langle G(s,u^0(s),z),e_i\rangle$ 
belongs to $L^2(\nu_T)$ by (H2). The weak convergence $\varphi_n\rightharpoonup\varphi$ therefore implies $\langle\gamma_n^k(t),e_i\rangle\to0$. Thus $\gamma_n^k(t)\to0$ in $H_k$ for every $t$. The Arzelà--Ascoli theorem now yields convergence of the entire sequence,
\[
\gamma_n^k\longrightarrow0\quad\text{ in }C([0,T];H_k), \qquad \text{ as } n\rightarrow\infty. 
\]
Since all norms on $H_k$ are equivalent and $H_k\subseteq V$, the convergence also holds in $C([0,T];H)\cap L^2([0,T];V)$.

Moreover, $\frac{\dif }{\dif t}\gamma_n^k(t)\in L^1([0,T];H)$, since
\begin{equation*}
	\int_{0}^{T}\left|\frac{\dif }{\dif t}\gamma_n^k(t)\right|\dif t 
	\leq \left(\int_{0}^{T}\int_{Z}|G(t,u^0(t),z)|^2\nu(\dif z)\dif t\right)^{\frac12}\|\varphi_n-\varphi\|_2<\infty.
\end{equation*}
The integration-by-parts formula therefore gives
%
%
\begin{eqnarray*}
	\langle \gamma_n^k(t), Y^{\varphi_n}(t)-Y^\varphi(t) \rangle &= &\int_{0}^{t}\left\langle\frac{\dif }{\dif s}\gamma_n^k(s),Y^{\varphi_n}(s)-Y^\varphi(s) \right\rangle\dif s \nonumber\\
	&&+ \int_{0}^{t}\left\langle\frac{\dif }{\dif s}\left(Y^{\varphi_n}(s)-Y^\varphi(s)\right),\gamma_n^k(s) \right\rangle\dif s.
\end{eqnarray*}
Therefore,
\begin{eqnarray}\label{MDP1-2}
	&&\int_{0}^{t}\int_{Z} \langle P_kG(s,u^0(s),z)(\varphi_n(s,z)-\varphi(s,z)), Y^{\varphi_n}(s)-Y^\varphi(s) \rangle\nu(\dif z)\dif s\nonumber\\
	&= & \langle\gamma_n^k(t), Y^{\varphi_n}(t)-Y^\varphi(t) \rangle-\int_{0}^{t}\left\langle\frac{\dif }{\dif s}\left(Y^{\varphi_n}(s)-Y^\varphi(s)\right),\gamma_n^k(s) \right\rangle\dif s \nonumber\\
	&= &  \langle\gamma_n^k(t), Y^{\varphi_n}(t)-Y^\varphi(t) \rangle 
	+ \int_{0}^{t}\left\langle \cA(Y^{\varphi_n}(s)-Y^\varphi(s)),\gamma_n^k(s) \right\rangle\dif s \nonumber\\
	&& + \int_{0}^{t}\Big[\left\langle B(Y^{\varphi_n}(s)-Y^\varphi(s),u^0(s)),\gamma_n^k(s) \right\rangle +\left\langle B(u^0(s),Y^{\varphi_n}(s)-Y^\varphi(s)),\gamma_n^k(s) \right\rangle\Big]\dif s \nonumber\\
	&& -\int_{0}^{t}\int_{Z} \left\langle G(s,u^0(s),z)(\varphi_n(s,z)-\varphi(s,z)),\gamma_n^k(s) \right\rangle\nu(\dif z)\dif s \nonumber\\
	&=: &I_1+I_2+I_3+I_4.
\end{eqnarray}
We estimate these terms as follows.
\begin{eqnarray}\label{MDP1-3}
	I_1&\leq &\frac{1}{4}|Y^{\varphi_n}(t)-Y^\varphi(t)|^2+C|\gamma_n^k(t)|^2, \nonumber\\
	I_2&\leq &\left(\int_{0}^{t}\|\cA(Y^{\varphi_n}(s)-Y^\varphi(s))\|_{V'}^2\dif s\right)^{\frac{1}{2}}\left(\int_{0}^{t}\|\gamma_n^k(s)\|^2\dif s\right)^{\frac{1}{2}} \nonumber\\
	&\leq & \sqrt{2}\left(\int_{0}^{t} \big( \|Y^{\varphi_n}(s)\|^2+\|Y^\varphi(s)\|^2 \big)\dif s\right)^{\frac{1}{2}}\left(\int_{0}^{t}\|\gamma_n^k(s)\|^2\dif s\right)^{\frac{1}{2}} \nonumber\\
	&\leq & \sqrt{2c_1}\left(\int_{0}^{t}\|\gamma_n^k(s)\|^2\dif s\right)^{\frac{1}{2}}, \nonumber\\
	I_4&\leq & \sqrt{2}m\sup_{s\in [0,t]}|\gamma_n^k(s)|\left(\int_{0}^{t}\int_{Z}|G(s,u^0(s),z)|^2\nu(\dif z)\dif s\right)^{\frac{1}{2}}\leq \sqrt{2c_2}m\sup_{s\in [0,t]}|\gamma_n^k(s)|,
\end{eqnarray}
where $c_2$ is given by 
\begin{equation*}\label{esti intG}
	\int_{0}^{T}\int_{Z}|G(t,u^0(t),z)|^2\nu(\dif z)\dif t 
	\leq  2\left(\|L_2\|_2^2+\|L_3\|_2^2 |u_0|^2\right)
	=: c_2.
\end{equation*}
As for $I_3$, consider
\begin{eqnarray*}
    &&\int_{0}^{t}\left\langle B(Y^{\varphi_n}(s)-Y^\varphi(s),u^0(s)),\gamma_n^k(s) \right\rangle \dif s \nonumber\\
    &\leq & Ca_0\int_{0}^{t}|Y^{\varphi_n}(s)-Y^\varphi(s)|^{\frac{1}{2}}\|Y^{\varphi_n}(s)-Y^\varphi(s)\|^{\frac{1}{2}}\|u^0(s)\||\gamma_n^k(s)|^{\frac{1}{2}}\|\gamma_n^k(s)\|^{\frac{1}{2}} \dif s \nonumber\\
    &\leq & Ca_0 \sup_{s\in [0,t]}|Y^{\varphi_n}(s)-Y^\varphi(s)|^{\frac{1}{2}}\sup_{s\in [0,t]}|\gamma_n^k(s)|^{\frac{1}{2}}\nonumber\\
    && \times\left(\int_{0}^{t} \|Y^{\varphi_n}(s)-Y^\varphi(s)\|^2 \dif s\right)^{\frac{1}{4}}\left(\int_{0}^{t} \|u^0(s)\|^2 \dif s\right)^{\frac{1}{2}}\left(\int_{0}^{t} \|\gamma_n^k(s)\|^2 \dif s\right)^{\frac{1}{4}} \nonumber\\
    &\leq & C|u_0|\sqrt{c_1}\left[ \sup_{t\in [0,T]}|\gamma_n^k(t)|+\left(\int_{0}^{t}\|\gamma_n^k(s)\|^2\dif s\right)^{\frac{1}{2}}\right],
\end{eqnarray*}
where we have used \eqref{p-5}, \eqref{p-4}, H\"older's inequality, and the estimates \eqref{esti X0} and \eqref{esti sk}. The other term in $I_3$ is estimated similarly, and hence
\begin{equation}\label{MDP1-4}
	I_3\leq C|u_0|\sqrt{c_1}\left[ \sup_{t\in [0,T]}|\gamma_n^k(t)|+\left(\int_{0}^{t}\|\gamma_n^k(s)\|^2\dif s\right)^{\frac{1}{2}}\right].
\end{equation}
For the projection error, the Cauchy--Schwarz and Young inequalities give
\begin{align*}
&\int_0^t\int_Z
\big|\langle (I-P_k)G(s,u^0(s),z)(\varphi_n(s,z)-\varphi(s,z)),
Y^{\varphi_n}(s)-Y^\varphi(s)\rangle\big|\,\nu(\dif z)\dif s\\
&\quad\leq \frac12\int_0^t\int_Z|(I-P_k)G(s,u^0(s),z)|^2\nu(\dif z)\dif s\\
&\qquad+\frac12\int_0^t\int_Z|\varphi_n(s,z)-\varphi(s,z)|^2
|Y^{\varphi_n}(s)-Y^\varphi(s)|^2\nu(\dif z)\dif s.
\end{align*}
Substituting this bound and \eqref{MDP1-1}--\eqref{MDP1-4} into \eqref{MDP1-0}, we obtain
\begin{eqnarray*}
	&&\frac{1}{4}|Y^{\varphi_n}(t)-Y^\varphi(t)|^2+\frac{1}{2}\int_{0}^{t}\|Y^{\varphi_n}(s)-Y^\varphi(s)\|^2\dif s \nonumber\\
	&\leq & C_2\left[\sup_{t\in [0,T]}|\gamma_n^k(t)|+\left(\int_{0}^{t}\|\gamma_n^k(s)\|^2\dif s\right)^{\frac{1}{2}} \right]\nonumber\\
	&& + \int_{0}^{t}\int_{Z}|(I-P_k)G(s,u^0(s),z)|^2\nu(\dif z)\dif s + \int_{0}^{t}\int_{Z}|\varphi_n(s,z)-\varphi(s,z)|^2|Y^{\varphi_n}(s)-Y^\varphi(s)|^2\nu(\dif z)\dif s \nonumber\\
	&&+ C_1\int_{0}^{t}\|u^0(s)\|^2|Y^{\varphi_n}(s)-Y^\varphi(s)|^2\dif s,\ t\in[0,T].
\end{eqnarray*}
Applying Gronwall's inequality yields, for every $t\in[0,T]$,
\begin{eqnarray*}
	&&|Y^{\varphi_n}(t)-Y^\varphi(t)|^2 \nonumber\\
	&\leq & \widetilde{C}_2\left\{\left[\sup_{t\in [0,T]}|\gamma_n^k(t)|+\left(\int_{0}^{t}\|\gamma_n^k(s)\|^2\dif s\right)^{\frac{1}{2}} \right] +\int_{0}^{t}\int_{Z}|(I-P_k)G(s,u^0(s),z)|^2\nu(\dif z)\dif  s\right\} \nonumber\\
	&& \times \exp\left\{\widetilde{C}_1\left(\int_{0}^{T}\|u^0(s)\|^2\dif s+\int_{0}^{T}\int_{Z}|\varphi_n(s,z)-\varphi(s,z)|^2\nu(\dif z)\dif s\right)\right\}.
\end{eqnarray*}
By the boundedness $\{\varphi_n\}$ and $\varphi$ in $L^2(\nu_T)$, the fact that $\gamma_n^k\rightarrow 0$ in $C([0,T];H)\cap L^2([0,T];V)$ as $n\to\infty$, and the convergence that
\begin{equation*}
\int_{0}^{T}\int_{Z}|(I-P_k)G(s,u^0(s),z)|^2\nu(\dif z)\dif s \longrightarrow0
\qquad\text{as }k\to\infty,
\end{equation*} 
we first letting $n\to\infty$ and then $k\to\infty$ to obtain
\[
\sup_{t\in[0,T]}|Y^{\varphi_n}(t)-Y^\varphi(t)|^2
+\int_0^T\|Y^{\varphi_n}(s)-Y^\varphi(s)\|^2\dif s\longrightarrow0, \qquad\text{ as } n\to\infty,
\]
which in particular implies the convergence in the metric of $D([0,T];H)\cap L^2([0,T];V)$.
\end{proof}

\subsection{Verification of (C2)}\label{S:5}
This subsection is devoted to the verification of (C2). The following lemma, taken from  \cite[Lemma 4.3]{BDG16} and \cite[Lemma 10.18]{BD19}, provides uniform estimates for  $\varphi_\varepsilon$ and $\psi_\varepsilon$.
\begin{lemma}
	Fix $m\in(0,\infty)$.
	\begin{itemize}
		\item[(a)] There exist $\Gamma_m, \rho_m,\xi_m:(0, \infty) \rightarrow(0, \infty)$ such that $\Gamma_m(s) \to 0$ as $s \to \infty$, $\xi_m(s)\to 0$ as $s\to 0$, and for all $I \in \mathcal{B}([0, T])$ and $\varepsilon, \beta \in(0, \infty)$,
		\begin{align}
		    &\sup _{\varphi \in S_\varepsilon^m} \int_{I \times Z}\left(L_1(s,z)+L_2(s,z)+L_3(s,z)\right)|\varphi(s, z)| \mathbf{1}_{\{|\varphi| \geqslant \beta / a(\varepsilon)\}}(s,z) \nu(\mathrm{d} z) \mathrm{d} s \nonumber\\
		      &\qquad \leq \sqrt{a(\varepsilon)}\Gamma_m(\beta) +(1+{\rm Leb}_{ \rm T}(I))\xi_m(\varepsilon), \label{coef2}\\
	        &\sup _{\varphi \in S_\varepsilon^m} \int_{I \times Z}\left(L_1(s,z)+L_2(s,z)+L_3(s,z)\right)|\varphi(s,z)|  \nu(\mathrm{d} z) \mathrm{d} s \nonumber\\
	        &\qquad \leq \rho_m(\beta)\sqrt{{\rm Leb}_{ \rm T}(I)}+\sqrt{a(\varepsilon)}\Gamma_m(\beta) +(1+{\rm Leb}_{ \rm T}(I))\xi_m(\varepsilon). \label{coef3}
		\end{align}
		\item[(b)] There exists $\zeta_m\in (0,\infty)$ such that for all $I\in\mathcal{B}([0,T])$ and $\varepsilon\in(0,\infty)$,
		\begin{equation}\label{coef1}
			\sup_{\psi\in S^m_{+,\varepsilon}}\int_{I\times Z}\left(L_1^2(s,z)+L_2^2(s,z)+L_3^2(s,z)\right)\psi(s,z)\nu(\dif z)\dif s\leq \zeta_m(a^2(\varepsilon)+{\rm Leb}_{\rm T}(I)).
		\end{equation}
	\end{itemize}
\end{lemma}
Recall the controlled SPDE \eqref{eq MDP 1-second} satisfied by $M^{\psi_\varepsilon}=\Gamma^{\varepsilon}(\varepsilon N^{\varepsilon^{-1}\psi_\varepsilon})$. We first establish a uniform estimate for this process.
\begin{lemma}\label{estiMq}
	For any $m\in(0,\infty)$, let $\{\psi_\varepsilon\}_{\varepsilon>0}$ be such that for every $\varepsilon>0$,
	$\psi_\varepsilon\in \mathcal{S}^m_{+,\varepsilon}$. Then there exists $\varepsilon_0>0$ such that 
	$$
	\sup_{\varepsilon\in (0,\varepsilon_0] }\mathbb{E} \left[\sup_{t\in[0,T]}|M^{\psi_\varepsilon}(t)|^2+\int_{0}^{T}\|M^{\psi_\varepsilon}(s)\|^2\dif s\right]<\infty.	$$
\end{lemma}
\begin{proof}
	By equation \eqref{eq MDP 1-second} and It$\mathrm{\hat{o}}$'s formula, for $t\in[0,T]$ we have
	\begin{eqnarray}\label{Mq0}
		&&\frac{1}{2}|M^{\psi_\varepsilon}(t)|^2+\int_0^t\| M^{\psi_\varepsilon}(s)\|^2\dif s\nonumber\\
		&=& -\frac{1}{a(\varepsilon)}\int_{0}^{t}\langle B\left(a(\varepsilon)M^{\psi_\varepsilon}(s)+u^0(s),a(\varepsilon)M^{\psi_\varepsilon}(s)+u^0(s)\right)-B(u^0(s),u^0(s)), M^{\psi_\varepsilon}(s)\rangle\dif s\nonumber\\
		&&+\frac{\varepsilon}{a(\varepsilon)}\int_{0}^{t}\int_{Z}\langle G(s,a(\varepsilon)M^{\psi_\varepsilon}(s-)+u^0(s-),z),M^{\psi_\varepsilon}(s-)\rangle\widetilde{N}^{\varepsilon^{-1}\psi_\varepsilon}(\dif z,\dif s)\nonumber\\
		&&+\frac{\varepsilon^2}{2a^2(\varepsilon)}\int_{0}^{t}\int_{Z} |G(s,a(\varepsilon)M^{\psi_\varepsilon}(s-)+u^0(s-),z)|^2 N^{\varepsilon^{-1}\psi_\varepsilon}(\dif z,\dif s)\nonumber\\
		&&+\frac{1}{a(\varepsilon)}\int_{0}^{t}\int_{Z}\langle G(s,a(\varepsilon)M^{\psi_\varepsilon}(s-)+u^0(s-),z)
		(\psi_\varepsilon(s,z)-1),M^{\psi_\varepsilon}(s)\rangle\nu(\dif z)\dif s \nonumber\\
		&=:&J_1+J_2+J_3+J_4.
	\end{eqnarray}
	By the bilinearity of the operator $B$, \eqref{p-4}-\eqref{eq PB} and Young's inequality we have
	\begin{equation}\label{Mq1}
		J_1	=-\int_{0}^{t} \langle B\left(M^{\psi_\varepsilon}(s),u^0(s)\right),M^{\psi_\varepsilon}(s) \rangle\dif s 
		\leq \frac{1}{2}\int_{0}^{t}\|M^{\psi_\varepsilon}(s)\|^2 \dif s+C\int_{0}^{t}|M^{\psi_\varepsilon}(s)|^2\|u^0(s)\|^2\dif s.
	\end{equation}
Recall that $\varphi_\varepsilon=(\psi_\varepsilon-1)/a(\varepsilon)$. By (H2), \eqref{esti X0}, and the elementary equality $2a\leq a^2+1$, we have
\begin{eqnarray}\label{Mq4}
	J_4&\leq& \int_{0}^{t}\int_{Z}\left(L_2(s,z)+L_3(s,z)|a(\varepsilon)M^{\psi_\varepsilon}(s)+u^0(s)|\right)|\varphi_\varepsilon(s,z)||M^{\psi_\varepsilon}(s)|\nu(\dif z)\dif s\nonumber\\
	&\leq& C\int_{0}^{t}\int_{Z}(L_2(s,z)+L_3(s,z))|\varphi_\varepsilon(s,z)||M^{\psi_\varepsilon}(s)|\nu(\dif z)\dif s \nonumber\\
	&&+\int_{0}^{t}\int_{Z}a(\varepsilon)L_3(s,z)|\varphi_\varepsilon(s,z)||M^{\psi_\varepsilon}(s)|^2\nu(\dif z)\dif s\nonumber\\
	&\leq & C\int_{0}^{t}\int_{Z}(L_2(s,z)+L_3(s,z))|\varphi_\varepsilon(s,z)|\nu(\dif z)\dif s \nonumber\\
	&&+C\int_{0}^{t}\int_{Z}\left(L_2(s,z)+L_3(s,z)\right)|\varphi_\varepsilon(s,z)||M^{\psi_\varepsilon}(s)|^2\nu(\dif z)\dif s.
\end{eqnarray}
Substituting \eqref{Mq1}--\eqref{Mq4} into \eqref{Mq0} and applying Gronwall's inequality, we obtain
\begin{eqnarray*}
	\frac{1}{2}|M^{\psi_\varepsilon}(t)|^2+\frac{1}{2}\int_0^t\| M^{\psi_\varepsilon}(s)\|^2\dif s\leq e^{CD_\varepsilon}(CD_\varepsilon+\sup_{t\in[0,T]}|J_2(t)+J_3(t)|),\ t\in[0,T],
\end{eqnarray*}
for some constant 
\begin{equation*}
	D_\varepsilon\leq|u_0|^2+\int_{0}^{T}\int_{Z}\left(L_2(s,z)+L_3(s,z)\right)|\varphi_\varepsilon(s,z)|\nu(\dif z)\dif s<\infty,
\end{equation*}
due to the fact that $L_2,L_3\in L^2(\nu_T)$ and \eqref{coef3}.
Then there exists a constant $\Lambda\geq 0$ such that
\begin{eqnarray}\label{Mq5}
	\mathbb{E}\left[\sup_{t\in[0,T]}|M^{\psi_\varepsilon}(t)|^2+\int_{0}^{T}\| M^{\psi_\varepsilon}(t)\|^2\dif t\right]\leq \Lambda\left(1+\mathbb{E}\left[\sup_{t\in[0,T]}\left|J_2(t)+J_3(t)\right|\right]\right),
\end{eqnarray}
By the Burkholder-Davis-Gundy inequality and (H2), $\mathbb{E}\left[\sup_{t\in[0,T]}|J_2(t)+J_3(t)|\right]$ is bounded by
\begin{eqnarray}\label{Mq23}
	&&\frac{\varepsilon C}{a(\varepsilon)}\mathbb{E}\left(\int_{0}^{T}\int_{Z}|G(s,a(\varepsilon) M^{\psi_\varepsilon}(s-)+u^0(s),z)|^2|M^{\psi_\varepsilon}(s-)|^2N^{\varepsilon^{-1}\psi_\varepsilon}(\dif z,\dif s)\right)^{\frac{1}{2}}\nonumber\\
	&&+\frac{\varepsilon C}{a^2(\varepsilon)}\mathbb{E}\left[\int_{0}^{T}\int_{Z}|G(s,a(\varepsilon) M^{\psi_\varepsilon}(s)+u^0(s),z)|^2\psi_\varepsilon(s,z)\nu(\dif z)\dif s\right]\nonumber\\
	&\leq & \frac{1}{4}\mathbb{E}\left[\sup_{s\in [0,T]}|M^{\psi_\varepsilon}(s)|^2\right]+\frac{\varepsilon C}{a^2(\varepsilon)}\mathbb{E}\left[\int_{0}^{T}\int_{Z}|G(s,a(\varepsilon) M^{\psi_\varepsilon}(s)+u^0(s),z)|^2\psi_\varepsilon(s,z)\nu(\dif z)\dif s\right]\nonumber\\
	&\leq & \frac{1}{4}\mathbb{E}\left[\sup_{s\in [0,T]}|M^{\psi_\varepsilon}(s)|^2\right]\nonumber\\
	&&+\frac{\varepsilon C}{a^2(\varepsilon)}\mathbb{E}\left[\int_{0}^{T}\int_{Z}\left(L_2^2(s,z)+L_3^2(s,z)|a(\varepsilon) M^{\psi_\varepsilon}(s)+u^0(s)|^2\right)\psi_\varepsilon(s,z)\nu(\dif z)\dif s\right]\nonumber\\
	&\leq & \frac{1}{4}\mathbb{E}\left[\sup_{s\in [0,T]}|M^{\psi_\varepsilon}(s)|^2\right]+\varepsilon C\sup_{\psi\in S^m_{+,\varepsilon}}\left(\int_{0}^{T}\int_{Z}L_3^2(s,z)\psi(s,z)\nu(\dif z)\dif s\right) ~\mathbb{E}\left[\sup_{s\in [0,T]}|M^{\psi_\varepsilon}(s)|^2\right]\nonumber\\
	&&+\frac{\varepsilon C}{a^2(\varepsilon)}\sup_{\psi\in S^m_{+,\varepsilon}}\left(\int_{0}^{T}\int_{Z}\left(L_2^2(s,z)+L_3^2(s,z)\right)\psi(s,z)\nu(\dif z)\dif s\right).
\end{eqnarray}
Combining \eqref{Mq23} with \eqref{Mq5}, we obtain
\begin{eqnarray*}
	&&\left(\frac{3}{4}-\varepsilon C\sup_{\psi\in S^m_{+,\varepsilon}}\int_{0}^{T}\int_{Z}L_3^2(s,z)\psi(s,z)\nu(\dif z)\dif s\right)\mathbb{E}\left[\sup_{s\in [0,T]}|M^{\psi_\varepsilon}(s)|^2\right]\nonumber\\
	&\leq & C\left(1+\sup_{\psi\in S^m_{+,\varepsilon}}\int_{0}^{T}\int_{Z}\left(L_2^2(s,z)+L_3^2(s,z)\right)\psi(s,z)\nu(\dif z)\dif s\right)
\end{eqnarray*}
In view of \eqref{coef1} and $\varepsilon/a^2(\varepsilon)\to0$, for all sufficiently small $\varepsilon$, say $\varepsilon\in(0,\varepsilon_0]$, we have
\begin{equation*}
	\sup_{\varepsilon\in (0,\varepsilon_0] }\mathbb{E} \left[\sup_{t\in[0,T]}|M^{\psi_\varepsilon}(t)|^2+\int_{0}^{T}\|M^{\psi_\varepsilon}(s)\|^2\dif s\right]<\infty.
\end{equation*}
\end{proof}

\begin{proposition}\label{MDP2}
	For any $m\in(0,\infty)$, let $\{\psi_\varepsilon\}_{\varepsilon>0}$ be such that for every $\varepsilon>0$,
	$\psi_\varepsilon\in \mathcal{S}^m_{+,\varepsilon}$, and
	for some $\beta\in(0,1]$, $\varphi_\varepsilon\mathbf{1}_{\{|\varphi_\varepsilon|\leq \beta/a(\varepsilon)\}}\in B_2(\sqrt{m\kappa_2(1)})$
	where $\varphi_\varepsilon=(\psi_\varepsilon-1)/a(\varepsilon)$. 
	Then  
\begin{align*}
\Gamma^\varepsilon(\varepsilon N^{\varepsilon^{-1}\psi_\varepsilon})-\Gamma^0(\varphi_\varepsilon\mathbf{1}_{\{|\varphi_\varepsilon|\leq \beta/a(\varepsilon)\}}) \to 0 \quad \text{ in probability} 
\end{align*}
with respect to the metric of $D([0,T];H)\cap L^2([0,T];V)$ as $\varepsilon\to 0$.

\end{proposition}
\begin{proof}
For brevity, write $\widetilde{Y}^\varepsilon$ for $Y^{\varphi_\varepsilon\mathbf{1}_{\{|\varphi_\varepsilon|\leq \beta/a(\varepsilon)\}}}$. Let $Z^\varepsilon(t):=M^{\psi_\varepsilon}(t)-\widetilde{Y}^\varepsilon(t) $, $t\in [0,T]$. 
To verify the condition \textbf{(C2)}, it suffices to prove that $Z^\varepsilon\rightarrow 0$ in probability with respect to the metric of $D([0,T];H)\cap L^2([0,T];V)$.
	For each $n\in\mathbb{N}$  and  $\varepsilon>0$, define the stopping time
	\begin{equation*}
		\tau_n^\varepsilon:=\inf\left\{t\geq 0: |M^{\psi_\varepsilon}(t)|^2+\int_{0}^{t}\|M^{\psi_\varepsilon}(s)\|^2\dif s\geq n\right\}\wedge T.
	\end{equation*} 
By Lemma \ref{estiMq} and Chebyshev's inequality, we have
\begin{equation*}
	\sup_{\varepsilon\in(0,\varepsilon_0]} \mathbb{P}(\tau_n^\varepsilon <T)\leq \sup_{\varepsilon\in(0,\varepsilon_0]} \mathbb{E}\left[\sup_{t\in[0,T]}|M^{\psi_\varepsilon}(t)|^2+\int_{0}^{T}\|M^{\psi_\varepsilon}(s)\|^2\dif s\right]\Big/n \leq C/n.
\end{equation*}

By It\^o's formula, we have
\begin{eqnarray}\label{Q0}
	&&\frac{1}{2}|Z^\varepsilon(t\wedge \tau_n^\varepsilon)|^2+\int_{0}^{t\wedge\tau_n^\varepsilon}\|Z^\varepsilon(s)\|^2\dif s\nonumber\\
	&=& -\int_{0}^{t\wedge \tau_n^\varepsilon} \left<\frac{B(a(\varepsilon)M^{\psi_\varepsilon}(s)+u^0(s),a(\varepsilon)M^{\psi_\varepsilon}(s)+u^0(s))-B(u^0(s),u^0(s))}{a(\varepsilon)}\right.\nonumber\\
	&&\left.-B(\widetilde{Y}^\varepsilon(s),u^0(s))-B(u^0(s),\widetilde{Y}^\varepsilon(s)),Z^\varepsilon(s)\right>\dif s\nonumber\\
	&&+\frac{\varepsilon}{a(\varepsilon)}\int_{0}^{t\wedge\tau_n^\varepsilon}\int_{Z}\left<G(s,a(\varepsilon) M^{\psi_\varepsilon}(s)+u^0(s),z),Z^\varepsilon(s)\right>\widetilde{N}^{\varepsilon^{-1}\psi_\varepsilon}(\dif z,\dif s)\nonumber\\
	&&+\frac{\varepsilon^2}{2a^2(\varepsilon)}\int_{0}^{t\wedge\tau_n^\varepsilon}\int_{Z}|G(s,a(\varepsilon) M^{\psi_\varepsilon}(s)+u^0(s),z)|^2N^{\varepsilon^{-1}\psi_\varepsilon}(\dif z,\dif s)\nonumber\\
	&&+\int_{0}^{t\wedge\tau_n^\varepsilon}\int_{Z}\left<G(s,a(\varepsilon) M^{\psi_\varepsilon}(s)+u^0(s),z)\varphi_\varepsilon(s,z)\right.\nonumber\\
	&&\quad\quad\quad\quad\quad\left.-G(s,u^0(s),z)\varphi_\varepsilon(s,z) 1_{\{|\varphi_\varepsilon|\leq \beta/a(\varepsilon)\}},Z^\varepsilon(s)\right>\nu(\dif z)\dif s\nonumber\\
	&=&: Q_1+Q_2+Q_3+Q_4.
\end{eqnarray}
Since $\varphi_\varepsilon\mathbf{1}_{\{|\varphi_\varepsilon|\leq \beta/a(\varepsilon)\}}\in B_2(\sqrt{m\kappa_2(1)})$, Proposition \ref{thm solution} (iii) implies that there exists $\Omega_0\in \mathcal{F}$ with $P(\Omega_0)=1$ such that 
\begin{equation*}
\vartheta=\sup_{\varepsilon\in (0,\varepsilon_0] }\sup_{\omega\in\Omega_0,t\in[0,T]}|\widetilde{Y}^\varepsilon(t)(\omega)|<\infty,
\end{equation*}
where the constant $\vartheta$ depends on $m$.

By \eqref{eq PB}, \eqref{p-5} and \eqref{p-4}, using Young's inequality and the definition of $\tau_n^\varepsilon$ , we have
\begin{eqnarray}\label{Q1}
	Q_1&=& -\int_{0}^{t\wedge \tau_n^{\varepsilon}}\left<a(\varepsilon)B(M^{\psi_\varepsilon}(s),M^{\psi_\varepsilon}(s))+B(Z^\varepsilon(s),u^0(s)),Z^\varepsilon(s)\right>\dif s\nonumber\\
	&\leq & C\int_{0}^{t\wedge \tau_n^{\varepsilon}}\left(a(\varepsilon)|M^{\psi_\varepsilon}(s)|^{\frac{1}{2}}\|M^{\psi_\varepsilon}(s)\|^{\frac{3}{2}}|Z^\varepsilon(s)|^{\frac{1}{2}} \|Z^\varepsilon(s)\|^{\frac{1}{2}} +|Z^\varepsilon(s)|\|Z^\varepsilon(s)\|\|u^0(s)\|\right)\dif s\nonumber\\
	&\leq & a(\varepsilon)Cn^{\frac{1}{4}}(n^{\frac{1}{4}}+\vartheta^{\frac{1}{2}})\int_{0}^{t\wedge\tau_n^\varepsilon}\|M^{\psi_\varepsilon}(s)\|^{\frac{3}{2}}\|Z^\varepsilon(s)\|^{\frac{1}{2}}\dif s+\int_{0}^{t\wedge\tau_n^\varepsilon}|Z^\varepsilon(s)|\|Z^\varepsilon(s)\|\|u^0(s)\|\dif s\nonumber\\
	&\leq & \frac{1}{2}\int_{0}^{t\wedge\tau_n^\varepsilon} \|Z^\varepsilon(s)\|^2\dif s+a(\varepsilon)Cn^{\frac{1}{2}}(n^{\frac{1}{4}}+\vartheta^{\frac{1}{2}})\int_{0}^{t\wedge\tau_n^\varepsilon}\|M^{\psi_\varepsilon}(s)\|^2\dif s\nonumber\\
	&&+C\int_{0}^{t\wedge\tau_n^\varepsilon}|Z^\varepsilon(s)|^2\|u^0(s)\|^2\dif s\nonumber\\
	&\leq & \frac{1}{2}\int_{0}^{t\wedge\tau_n^\varepsilon} \|Z^\varepsilon(s)\|^2\dif s+a(\varepsilon)Cn^{\frac{5}{4}}(n^{\frac{1}{4}}+\vartheta^{\frac{1}{2}})+C\int_{0}^{t\wedge\tau_n^\varepsilon}|Z^\varepsilon(s)|^2\|u^0(s)\|^2\dif s.
\end{eqnarray}
By (H1) and (H2) we have
\begin{eqnarray}\label{Q4}
	Q_4&\leq &	\int_{0}^{t\wedge\tau_n^\varepsilon}\int_{Z}|G(s,a(\varepsilon)M^{\psi_\varepsilon}(s)+u^0(s),z)-G(s,u^0(s),z)||\varphi_\varepsilon(s,z)||Z^\varepsilon(s)|\nu(\dif z)\dif s\nonumber\\
	&&+\int_{0}^{t\wedge\tau_n^\varepsilon}\int_{Z} |G(s,u^0(s),z)| |\varphi_\varepsilon(s,z)|\mathbf{1}_{\{|\varphi_\varepsilon|> \beta/a(\varepsilon)\}}|Z^\varepsilon(s)|	\nu(\dif z)\dif s\nonumber\\
	&\leq & a(\varepsilon)n^{\frac{1}{2}}(n^{\frac{1}{2}}+\vartheta)\int_{0}^{t\wedge\tau_n^\varepsilon}\int_{Z}L_1(s,z)|\varphi_\varepsilon(s,z)|\nu(\dif z)\dif s\nonumber\\
	&&+C(n^{\frac{1}{2}}+\vartheta)\int_{0}^{t\wedge\tau_n^\varepsilon}\int_{Z} \left(L_2(s,z)+L_3(s,z)\right)|\varphi_\varepsilon(s,z)|\mathbf{1}_{\{|\varphi_\varepsilon|> \beta/a(\varepsilon)\}}\nu(\dif z)\dif s.
\end{eqnarray}
Substituting \eqref{Q1} and \eqref{Q4} into \eqref{Q0}, applying Gronwall's inequality, using \eqref{esti X0} and then taking expectations, we obtain
\begin{equation}\label{Q5}
	\mathbb{E}\sup_{t\in[0,T]}\left[|Z^\varepsilon(t\wedge \tau_n^\varepsilon)|^2+\int_{0}^{t\wedge\tau_n^\varepsilon}\|Z^\varepsilon(s)\|^2\dif s\right]
	\leq  \Lambda\left(\rho_\varepsilon+\mathbb{E}\sup_{t\in[0,T]}\left(|Q_2(t)|+|Q_3(t)|\right)\right),
\end{equation}
for some constant $\Lambda>0$, where 
\begin{align}\label{rho} \rho_\varepsilon=& \  a(\varepsilon)n^{\frac{5}{4}}(n^{\frac{1}{4}}+\vartheta^{\frac{1}{2}})+a(\varepsilon)n^{\frac{1}{2}}(n^{\frac{1}{2}}+\vartheta)\sup_{\varphi\in S^m_\varepsilon} \int_{0}^{T}\int_{Z}L_1(s,z)|\varphi(s,z)|\nu(\dif z)\dif s\nonumber\\
	&+(n^{\frac{1}{2}}+\vartheta)\sup_{\varphi\in S^m_\varepsilon} \int_{0}^{T}\int_{Z} \left(L_2(s,z)+L_3(s,z)\right)|\varphi(s,z)|\mathbf{1}_{\{|\varphi_\varepsilon|> \beta/a(\varepsilon)\}}	\nu(\dif z)\dif s.
\end{align}
Similar to \eqref{Mq23}, we have
\begin{eqnarray}\label{Q23}
	&&\Lambda\,\mathbb{E}\left[\sup_{t\in[0,T]}\left(|Q_2(t)|+|Q_3(t)|\right)\right]\nonumber\\
	&\leq & \frac{1}{2}\mathbb{E}\left[\sup_{t\in [0,T]}|Z^\varepsilon(t\wedge\tau_n^\varepsilon)|^2\right]+\varepsilon C\sup_{\psi\in S^m_{+,\varepsilon}}\left(\int_{0}^{T}\int_{Z}L_3^2(s,z)\psi(s,z)\nu(\dif z)\dif s\right) ~\mathbb{E}\left[\sup_{t\in [0,T]}|Z^\varepsilon(t\wedge\tau_n^\varepsilon)|^2\right]\nonumber\\
	&&+\frac{\varepsilon C}{a^2(\varepsilon)}\sup_{\psi\in S^m_{+,\varepsilon}}\left(\int_{0}^{T}\int_{Z}\left(L_2^2(s,z)+L_3^2(s,z)\right)\psi(s,z)\nu(\dif z)\dif s\right).
\end{eqnarray}
Combining \eqref{Q5} and \eqref{Q23} yields
\begin{eqnarray*}
	&&\mathbb{E}\left[\sup_{t\in[0,T]}\left(\left(\frac{1}{2}-\varepsilon C\sup_{\psi\in S^m_{+,\varepsilon}}\int_{0}^{T}\int_{Z}L_3^2(s,z)\psi(s,z)\nu(\dif z)\dif s \right)|Z^\varepsilon(t\wedge \tau_n^\varepsilon)|^2+\int_{0}^{t\wedge\tau_n^\varepsilon}\|Z^\varepsilon(s)\|^2\dif s\right)\right]\nonumber\\
	&\leq & \Lambda\rho_\varepsilon+\frac{\varepsilon C}{a^2(\varepsilon)}\sup_{\psi\in S^m_{+,\varepsilon}}\left(\int_{0}^{T}\int_{Z}\left(L_2^2(s,z)+L_3^2(s,z)\right)\psi(s,z)\nu(\dif z)\dif s\right),
\end{eqnarray*}
In view of \eqref{coef2}-\eqref{coef1} and \eqref{rho}, we obtain that
\begin{equation*}
 \lim_{\varepsilon\to 0}\mathbb{E}\left[\sup_{t\in[0,T]}\left(|Z^\varepsilon(t\wedge \tau_n^\varepsilon)|^2+\int_{0}^{t\wedge\tau_n^\varepsilon}\|Z^\varepsilon(s)\|^2\dif s\right)\right]=0.
\end{equation*}
For any $\delta>0$,
\begin{eqnarray*}
    &&\mathbb{P}\left(\sup_{t\in[0,T]}|Z^\varepsilon(t)|^2+\int_{0}^{T}\|Z^\varepsilon(s)\|^2\dif s>\delta\right)\nonumber\\
    &\leq & 
    \mathbb{P}\left(\sup_{t\in[0,T]}|Z^\varepsilon(t)|^2+\int_{0}^{T}\|Z^\varepsilon(s)\|^2\dif s>\delta, \tau_n^\varepsilon\geq T\right)+\mathbb{P}\left(\tau_n^\varepsilon< T\right)\nonumber\\
    &\leq & 
    \mathbb{P}\left(\sup_{t\in[0,T]}|Z^\varepsilon(t\wedge \tau_n^\varepsilon)|^2+\int_{0}^{t}\|Z^\varepsilon(s\wedge \tau_n^\varepsilon)\|^2\dif s>\delta, \tau_n^\varepsilon\geq T\right)+\mathbb{P}\left(\tau_n^\varepsilon< T\right)\nonumber\\
    &\leq &  \mathbb{E}\left[\sup_{t\in[0,T]}\left(|Z^\varepsilon(t\wedge \tau_n^\varepsilon)|^2+\int_{0}^{t\wedge\tau_n^\varepsilon}\|Z^\varepsilon(s)\|^2\dif s\right)\right]\big/\delta+C/n.
\end{eqnarray*}
Letting $\varepsilon\to 0$ and then $n\to \infty$, we obtain that
\begin{equation*}
	\lim_{\varepsilon\to 0}\mathbb{P}\left(\sup_{t\in[0,T]}|Z^\varepsilon(t)|^2+\int_{0}^{T}\|Z^\varepsilon(s)\|^2\dif s>\delta\right)=0.
\end{equation*}
The proof of Proposition \ref{MDP2} is completed.
\end{proof}

\end{document}